\documentclass[submission,copyright,creativecommons]{eptcs}
\providecommand{\event}{ACT 2022} 
\usepackage{breakurl}             
\usepackage{underscore}           
\usepackage[english]{babel}	

\usepackage{amssymb,amsmath,amsthm}       %
\usepackage{amsfonts}      %

\newtheorem{theorem}{Theorem}
\newtheorem{lemma}{Lemma}
\newtheorem{proposition}{Proposition}

\newtheorem{example}{Example}
\newtheorem{remark}{Remark}

\usepackage{mathtools}     %

\usepackage{verbatimbox}

\usepackage{tikz}
\usepackage[all]{xy}
\usetikzlibrary{arrows,shapes}
\usetikzlibrary{decorations.pathmorphing}
\usetikzlibrary{decorations.markings}

\title{Magnitude and Topological Entropy of Digraphs}
\author{Steve Huntsman
\institute{STR \\ Arlington, Virginia}
\email{steve.huntsman@str.us}
}
\def\titlerunning{Magnitude and Topological Entropy of Digraphs}
\def\authorrunning{S. Huntsman}
\begin{document}
\maketitle

\begin{abstract}
Magnitude and (co)weightings are quite general constructions in enriched categories, yet they have been developed almost exclusively in the context of Lawvere metric spaces. We construct a meaningful notion of magnitude for flow graphs based on the observation that topological entropy provides a suitable map into the max-plus semiring, and we outline its utility. Subsequently, we identify a separate point of contact between magnitude and topological entropy in digraphs that yields an analogue of volume entropy for geodesic flows. Finally, we sketch the utility of this construction for feature engineering in downstream applications with generic digraphs.
\end{abstract}

\section{Introduction}

Let $\mathbf{M} = (\mathbf{M},\otimes,1)$ be a monoidal category (for background, see \cite{mac1978categories,fong2019invitation}) and $\mathbf{C}$ a (small) $\mathbf{M}$-category, i.e., a (small) category enriched over $\mathbf{M}$. Recall that this means that $\mathbf{C}$ is specified by a set $\text{Ob}(\mathbf{C})$; hom-objects $\mathbf{C}(j,k) \in \mathbf{M}$ for all $j, k \in \text{Ob}(\mathbf{C})$; identity morphisms $1 \rightarrow \mathbf{C}(j,j)$ for all $j \in \text{Ob}(\mathbf{C})$; and composition morphisms $\mathbf{C}(j,k) \otimes \mathbf{C}(k,\ell) \rightarrow \mathbf{C}(j,\ell)$ for all $j, k, \ell \in \text{Ob}(\mathbf{C})$; moreover, these hom-objects and morphisms are required to satisfy associativity and unitality properties \cite{kelly1982basic,fong2019invitation}.

The theory of magnitude \cite{leinster2017magnitude,leinster2021entropy} incorporates a $\mathbf{M}$-category $\mathbf{C}$ and a semiring $S$ via a ``size'' map $\sigma : \text{Ob}(\mathbf{M}) \rightarrow S$ that is constant on isomorphism classes and that satisfies $\sigma(1) = 1$ and $\sigma(X \otimes Y) = \sigma(X) \cdot \sigma(Y)$, where the semiring unit and multiplication are indicated on the right-hand sides. If $n := |\text{Ob}(\mathbf{C})| < \infty$ then its \emph{similarity matrix} $Z \in M(n,S)$ has entries $Z_{jk} := \sigma(\mathbf{C}(j,k))$. Introducing the (common) notation $$(f[X])_{jk} := f(X_{jk})$$ as a shorthand where $X$ is a matrix over the semiring $S$ and $f$ is a function on $S$, we have $Z := \sigma[\mathbf{C}]$.

A \emph{weighting} is a column vector $w$ satisfying $Zw = 1$, where the semiring matrix multiplication and column vector of ones are indicated. A \emph{coweighting} is the transpose of a weighting for $Z^T$. If $Z$ has a weighting and a coweighting, its \emph{magnitude} is the sum of the components of either one of these: a line of algebra shows these sums necessarily coincide. 

The notion of magnitude has been the subject of increasing attention over the past 15 years, and over the last year or so applications have begun to emerge based on boundary-detecting properties of (co)weightings in the setting of metric spaces \cite{bunch2020practical,huntsman2022diversity}, which is virtually the only case that has been explored to date. 
\footnote{
The only exception of which we are aware is \cite{chuang2016magnitude}, which details a nontrivial example of magnitude for a certain $\textbf{Vect}$-category; see also Example 6.4.5 of \cite{leinster2021entropy}. 
}
This setting emerges from the choice $\mathbf{M} = (([0,\infty],\ge),+,0)$, which with only a very mild continuity assumption requires $\sigma(x) = \exp(-t x)$ for some constant $t$; varying this constant leads to the notion of a \emph{magnitude function}. The corresponding enriched categories are precisely the \emph{Lawvere metric spaces}, also known as \emph{extended quasipseudometric spaces} since they generalize metric spaces by allowing distances that are infinite (extended), asymmetric (quasi-), or zero (pseudo-). In \S \ref{sec:SimilarityArithmeticRigidity} we will show that seemingly adjacent monoidal structures on $([0,\infty],\ge)$ in fact lead to the same construction, so to move away from the generalized metric space setting at all, it is necessary to move quite far indeed.


However, there are other interesting monoidal categories that yield applicable instantiations of magnitude, though \S \ref{sec:SimilarityArithmeticRigidity} shows that these must necessarily give rise to something quite different from metric spaces. In \S \ref{sec:Flowgraphs}, we introduce such a construction via a monoidal category $\mathbf{Flow}$ of \emph{flow graphs} that informs the analysis of computer programs (and also, e.g., business processes), encompassing constructs that represent the transfer of control and data \cite{cooper2011engineering,nielson2004principles} as in Figure \ref{fig:example1}. This category has two monoidal products that model ``series'' and ``conditional'' (versus ``parallel'' \emph{per se}) execution of programs as well as the structure of an operad in $\mathbf{Set}$ \cite{huntsman2019multiresolution} that dovetails with a hierarchical representation of input/output structure \cite{johnson1994program}.

For each generic flow graph $D$, there is a $\mathbf{Flow}$-category described in Lemma \ref{lemma:enriched}. The \emph{topological entropy} of hom-objects in this category provides a suitable map $\sigma$ into the max-plus semiring, and the resulting weighting (resp., coweighting) indicate sub-flow graphs of maximal entropy in the ``forward direction'' (resp., ``reverse direction''). These constructions are attractive from the point of view of feature engineering for graph matching \cite{emmert2016fifty} and machine learning problems involving flow graphs. 

Meanwhile, once we consider interactions between magnitude and topological entropy in the setting of digraphs, another point of contact is readily discernible, and we discuss it in \S \ref{sec:Neighborhoods}. The magnitude function of a ball in the universal cover of a strong loopless digraph is closely related to the topological entropy of the digraph. In \S \ref{sec:Example} we provide evidence of the utility for feature engineering based on this observation in problems involving generic digraphs.

\section{\label{sec:SimilarityArithmeticRigidity}Rigidity of similarity matrix arithmetic}

Here we show that there is even less choice in how the theory of magnitude can be applied to metric spaces and their ilk than \S 2.3 of \cite{leinster2017magnitude} suggests, wherein the usual addition operation on $([0,\infty],\ge)$ is chosen for the monoidal structure. This rigidity illustrates that meaningful notions of magnitude outside its usual arena are likely to involve very different monoidal structures and/or categories.

\begin{proposition}
Let $f$ be a strictly increasing bijection from $[0,\infty]$ to a subset of $[-\infty,\infty]$ containing $0$. Then $x \otimes y := f^{-1}(f(x)+f(y))$ gives rise to a strict symmetric monoidal structure on $([0,\infty],\ge)$ with monoidal (additive) unit $f^{-1}(0)$. \qed
\end{proposition}

A category $\mathbf{C}$ enriched over the strict symmetric monoidal category above has, for every $j,k \in \text{Ob}(\mathbf{C})$, some $\eta_{jk} := \mathbf{C}(j,k) \in [0,\infty]$ such that $\eta_{jj} = f^{-1}(0)$ and $\eta_{jk} \otimes \eta_{k\ell} \ge \eta_{j\ell}$. That is, we have the triangle inequality $f(\eta_{jk})+f(\eta_{k\ell}) \ge f(\eta_{j\ell})$. Let us therefore assume $f[\eta] = d$, and furthermore stipulate that we want our similarity matrix $Z$ to take values in the semiring $\mathbb{R}$ with the usual structure, as opposed to some more exotic choice. Then we require a function $\sigma : [0,\infty] \rightarrow \mathbb{R}$ such that $\sigma(x \otimes y) = \sigma(x) \cdot \sigma(y)$ in order to define $Z := \sigma[\eta]$. If we require continuity, then this generalized Cauchy equation has the unique family of solutions $\sigma(x) = \exp(-\tau f(x))$ for $\tau \in \mathbb{R}$. Now $Z = \sigma[\eta] = \sigma[f^{-1}[d]] = \exp[-\tau d]$, just as usual: i.e., this attempted generalization actually has no material effect.

What about a more exotic semiring structure on $\mathbb{R}$? The proposition above has a close analogue:

\begin{proposition}
Let $g$ be a strictly increasing function from $[-\infty,\infty]$ to itself, and taking on the value $0$ (and also $1$ for the final part of the statement). Then $x \oplus y := g^{-1}(g(x)+g(y))$ gives rise to a strict symmetric monoidal structure on $([-\infty,\infty],\ge)$ with monoidal (additive) unit $g^{-1}(0)$. Moreover, additionally taking $x \odot y := g^{-1}(g(x) \cdot g(y))$ gives a semiring with multiplicative unit $g^{-1}(1)$.
\footnote{
We thank S. Tringali for this observation. If $g(x) := \text{sgn}(x) \cdot |x|^p$ for $p > 0$, we get the semiring $([-\infty,\infty],\oplus,0,\cdot,1)$. If $g(x) := \exp(-\tau x)$ for $\tau < 0$, then we get the semiring $([-\infty,\infty],\oplus,-\infty,+,0)$. 
}
\qed
\end{proposition}

Now the equation for a weighting is $\bigoplus_k (Z_{jk} \odot w_k) = g^{-1}(1)$, which unpacks to the matrix equation $g[Z]g[w] = 1$ in ordinary arithmetic. Recalling that $Z = \sigma[\eta]$ and $f[\eta] = d$, we have $Z = \sigma[f^{-1}[d]]$. Meanwhile, we have the generalized Cauchy equation $\sigma(x \otimes y) = \sigma(x) \odot \sigma(y)$, which unpacks to
\begin{equation}
\label{eq:generalizedCauchy}
\sigma(f^{-1}(f(x)+f(y))) = g^{-1}(g(\sigma(x)) \cdot g(\sigma(y))).
\end{equation}
Defining $h := g \circ \sigma \circ f^{-1}$, this becomes $h(f(x)+f(y)) = h(f(x)) \cdot h(f(y))$, i.e., $h$ satisfies the usual Cauchy equation; assuming continuity, we have $h[d] = \exp[-\tau d]$. Since $g[Z] = h[d]$, the weighting equation is $h[d]g[w] = 1$, which apart from the transformation of $w$ is the same as in ordinary arithmetic. 

In short, it appears to be at least difficult--perhaps impossible--to get substantially different arithmetic of similarity matrices than the ``default'' while still working over the extended real numbers, regardless of which underlying arithmetic we use. The thin silver lining is that we can legitimately apply a very broad class of componentwise transformations to a (co)weighting and still interpret the result as a (co)weighting also, albeit with respect to a different underlying semiring structure.

Nevertheless, the notion of magnitude still affords useful application to quite different monoidal categories; in the sequel, we give an example.

\section{\label{sec:Flowgraphs}Max-plus magnitude for flow graphs}

Throughout this paper, by \emph{digraph} we mean the usual notion in combinatorics. In particular, we do not allow multiple edges between vertices (i.e., a quiver is generally not a digraph \emph{per se}). See footnote \ref{foot:digraph}.

Consider the specific notion of \emph{flow graph} discussed in \cite{huntsman2019multiresolution}, viz. a digraph $D$ with exactly one source and exactly one target, such that there is a unique (entry) edge from the source and a unique (exit) edge to the target, and such that identifying the source of the entry edge with the target of the exit edge yields a strong digraph (i.e., a digraph in which every two vertices are connected by some path).
An example is the digraph in the right panel of Figure \ref{fig:example1}.

\begin{figure}
\begin{minipage}{0.3\textwidth}
\begin{verbnobox}[\scriptsize\arabic{VerbboxLineNo}\scriptsize\hspace{3ex}]
START
repeat
    repeat
        repeat
            if b goto 7
            if b
                repeat
                    S
                until b
            endif
        until b
        do while b
            do while b
                repeat
                    S
                until b
            enddo
        enddo
    until b
until b
HALT
\end{verbnobox}
\end{minipage}
\begin{minipage}{0.1\textwidth}
\ \\
\end{minipage}
\begin{minipage}{0.6\textwidth}
	\includegraphics[trim = 65mm 100mm 55mm 100mm, clip, width=\textwidth, keepaspectratio]{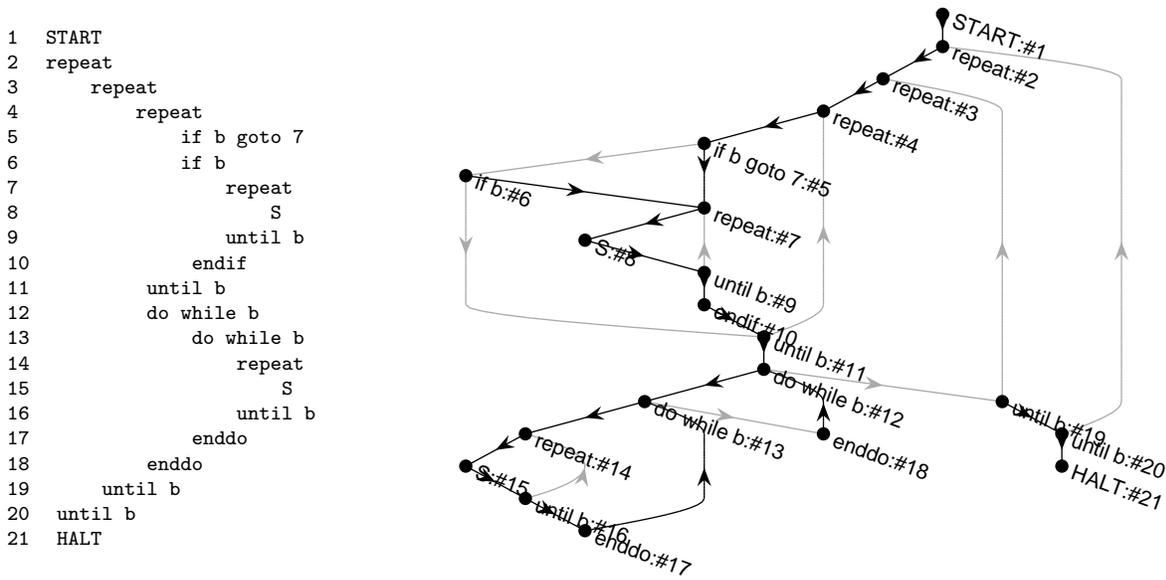}
\end{minipage}
	\caption{\label{fig:example1} (L) A simple imperative program. \texttt{S} denotes a generic statement (or subroutine); \texttt{b} denotes a generic Boolean predicate. (R) The corresponding control flow graph: branches are shaded {\color{black}black} (resp., {\color{gray}gray}) if the corresponding \texttt{b} evaluates to ${\color{black}\top}$ or ${\color{gray}\bot}$.}
\end{figure}

Let ${\bf Flow}$ be the full subcategory of reflexive digraphs
\footnote{
\label{foot:digraph}
An object in the category $\mathbf{Dgph}$ of reflexive digraphs is $G = (U,\alpha,\omega)$, where $U$ is a set and $\alpha, \omega : U \rightarrow U$ are \emph{head} and \emph{tail} functions that satisfy $\alpha \circ \omega = \omega$ and $\omega \circ \alpha = \alpha$. For $G' = (U',\alpha',\omega')$, a morphism $f \in {\bf Dgph}(G,G')$ is a function $f : U \rightarrow U'$ such that $f \circ \alpha = \alpha' \circ f$ and $f \circ \omega = \omega' \circ f$. The \emph{vertices} of $G = (U,\alpha,\omega)$ are the (mutual) image $V \equiv V(G)$ of $\alpha$ and $\omega$; the \emph{loops} are the set $L \equiv L(G) := \{u \in U : \alpha(u) = \omega(u)\}$ (so that $V \subseteq L$), and the \emph{edges} are the set $E \equiv E(G) := U \backslash L$. We recover the usual notion of a digraph by considering $\alpha \times \omega$ and its appropriate restrictions on $U^2$, $L^2$, and $E^2$: e.g., we can abusively write $E = (\alpha \times \omega)(E^2)$, where the LHS and RHS respectively refer to usual and reflexive notions of digraph edges. Thus a morphism $f : U \rightarrow U'$ restricts to $f|_V : V \rightarrow V'$, $f|_L : L \rightarrow L'$, and $f|_E : E \rightarrow U'$. Since morphisms are only partially specified by their actions on vertices, defining ${\bf Flow}$ as a full subcategory of $\mathbf{Dgph}$ is essentially a convention about vertex identification.
}
whose objects are (combinatorially realized as) flow graphs. It turns out that there are both ``series'' and ``parallel'' tensor products on ${\bf Flow}$, as well as the structure of an operad in $\mathbf{Set}$ which has a conceptually and algorithmically attractive instantiation. We are presently interested in the ``series'' tensor product, denoted $\boxtimes$. The idea of $\boxtimes$ is just to identify the exit edge of its first argument with the entry edge of its second argument (so unlike the ``parallel'' tensor product, this does not give rise to a symmetric monoidal structure). It turns out that this yields (the monoidal base of) an enriched category, viz. the ${\bf Flow}$-category ${\bf SubFlow}_D$ of sub-flow graphs of a flow graph $D$ (these correspond to subroutines in the context of program control flow).

\begin{lemma} \cite{huntsman2019multiresolution}
\label{lemma:enriched} 
For a flow graph $D$, we can form a category ${\bf SubFlow}_D$ enriched over ${\bf Flow}$ as follows:
\begin{itemize}
\item $\textnormal{Ob}({\bf SubFlow}_D) := E(D)$ (i.e., the objects of ${\bf SubFlow}_D$ are the edges of the digraph $D$); \footnote{Loops and reflexive self-edges are not included here, though the former may be accommodated without substantial changes.}
\item for $e_s, e_t \in {\bf SubFlow}_D$, the hom object ${\bf SubFlow}_D(e_s,e_t) \in {\bf Flow}$ is the (possibly empty) induced sub-flow graph of $D$ with entry edge $e_s$ and exit edge $e_t$: we denote this by $D\langle e_s, e_t \rangle$; 
\item the composition morphism is induced by $\boxtimes$; 
\item the identity element is determined by the flow graph $e$ with one edge. \qed
\end{itemize}
\end{lemma}


A digraph $D$ determines a \emph{(sub)shift of finite type}, i.e., a dynamical system on the space of paths in $D$ with an evolution operator that simply shifts path indices. The corresponding \emph{topological entropy} $h(D) := \lim_{N \uparrow \infty} N^{-1} \log W(D,N)$ measures the growth of the number $W(D,N)$ of paths in $D$ of length $N$ \cite{kitchens2012symbolic}. A basic result in symbolic dynamics is that $h(D)$ is given by the logarithm of the spectral radius of the adjacency matrix of $D$. 
(If $D$ is strong, the spectral radius is the Perron eigenvalue $\ge 1$.) 

\begin{lemma} 
\label{lemma:entropy} For $D_j \in {\bf Flow}$,
\begin{equation}
\label{eq:entropy}
h(\boxtimes_j D_j) = \max_j h(D_j).
\end{equation}
\end{lemma}

\begin{proof}
To see the $\ge$ direction, consider paths that are confined to whichever $D_j$ has highest topological entropy. For the $\le$ direction, note that the number of paths that are not so confined cannot grow at a faster rate. 
\end{proof}

\begin{remark}
In fact more is true: writing $A(D)$ for the adjacency matrix of $D$, we have via standard Perron-Frobenius theory that (as multisets) $$\textnormal{spec } A(\boxtimes_j D_j) = \{0\} \cup \bigcup_j \textnormal{spec } A(D_j).$$ (The zero is due to the first column/last row [using the obvious indexing] of $A(\boxtimes_j D_j)$ being identically zero.) Defining the zeta function $\zeta_D(t) := 1/\det(I-tA(D))$ \cite{mizuno2001zeta}, we furthermore have that $$\zeta_{\boxtimes_j D_j} = \prod_j \zeta_{D_j}.$$
For examples, see Figures \ref{fig:foo} and \ref{fig:baz}.

\begin{figure}[h]
  \centering
  \includegraphics[trim = 20mm 65mm 20mm 75mm, clip, width=.49\textwidth,keepaspectratio]{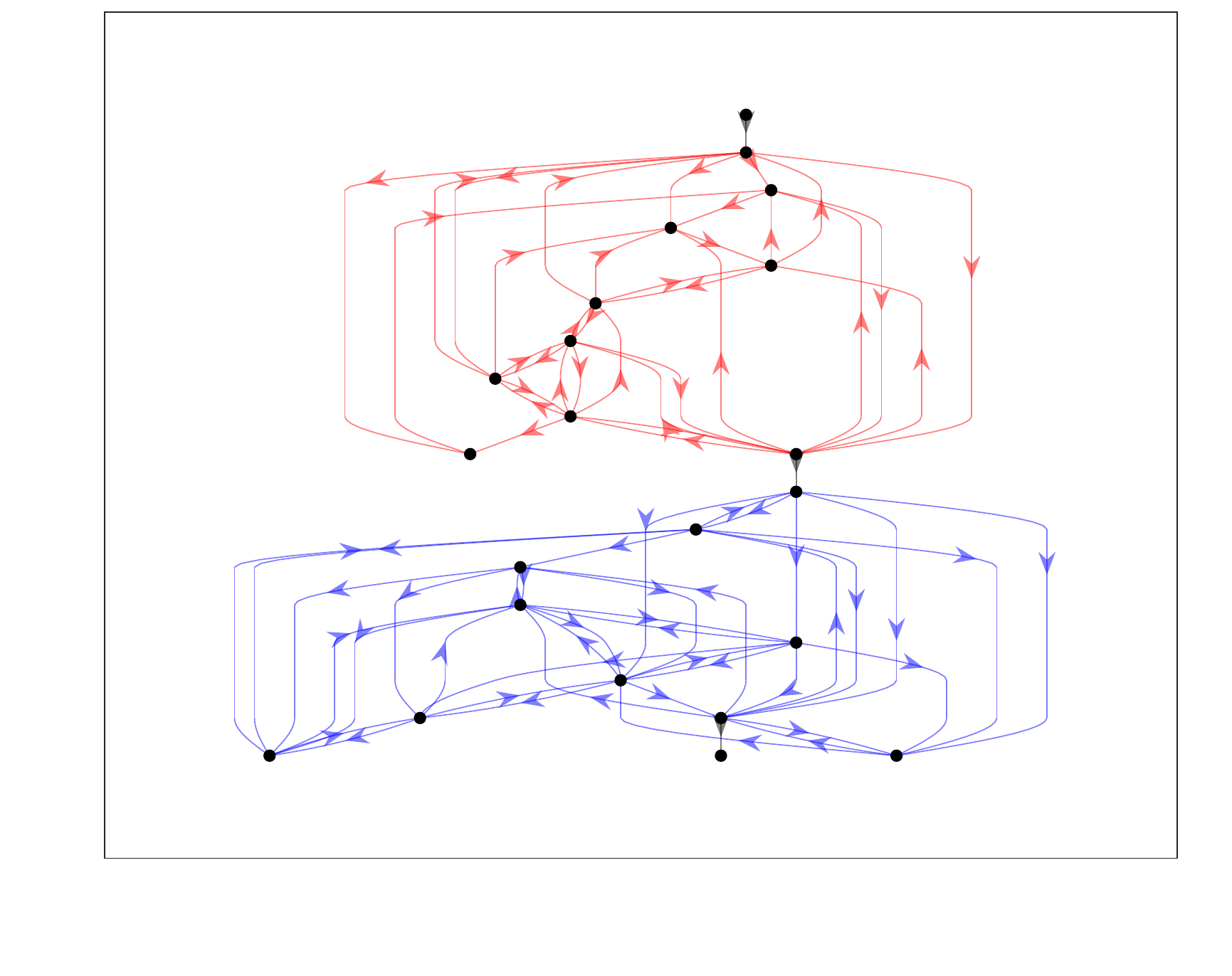}
  \includegraphics[trim = 20mm 65mm 20mm 75mm, clip, width=.49\textwidth,keepaspectratio]{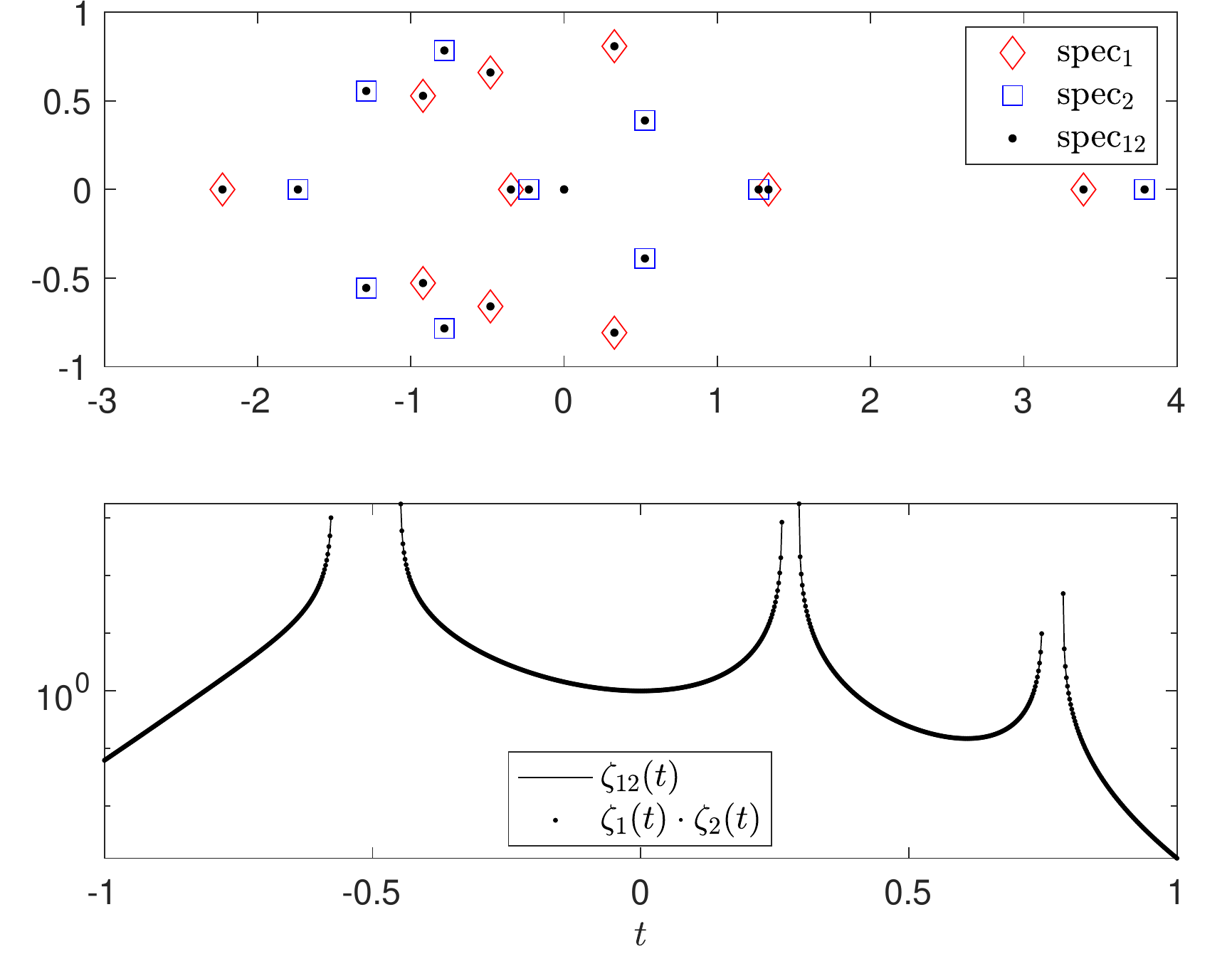} 
  \caption{\label{fig:foo} Left: ${\color{red}D_1} \boxtimes {\color{blue}D_2}$ for two flow graphs ${\color{red}D_1}$ and ${\color{blue}D_2}$ on 10 vertices. Upper right: spectra $\text{spec}_x \subset \mathbb{C}$ of the adjacency matrices $A(D_x)$ for ${\color{red}x = 1}$, ${\color{blue}x = 2}$, and $x = 12$ with $D_{12} := {\color{red}D_1} \boxtimes {\color{blue}D_2}$. Lower right: zeta functions $\zeta_{12}$ and $\zeta_{\color{red}1} \cdot \zeta_{\color{blue}2}$ with $\zeta_x \equiv \zeta_{D_x}$.}
\end{figure}


\begin{figure}[h]
  \centering
  \includegraphics[trim = 20mm 65mm 20mm 75mm, clip, width=.49\textwidth,keepaspectratio]{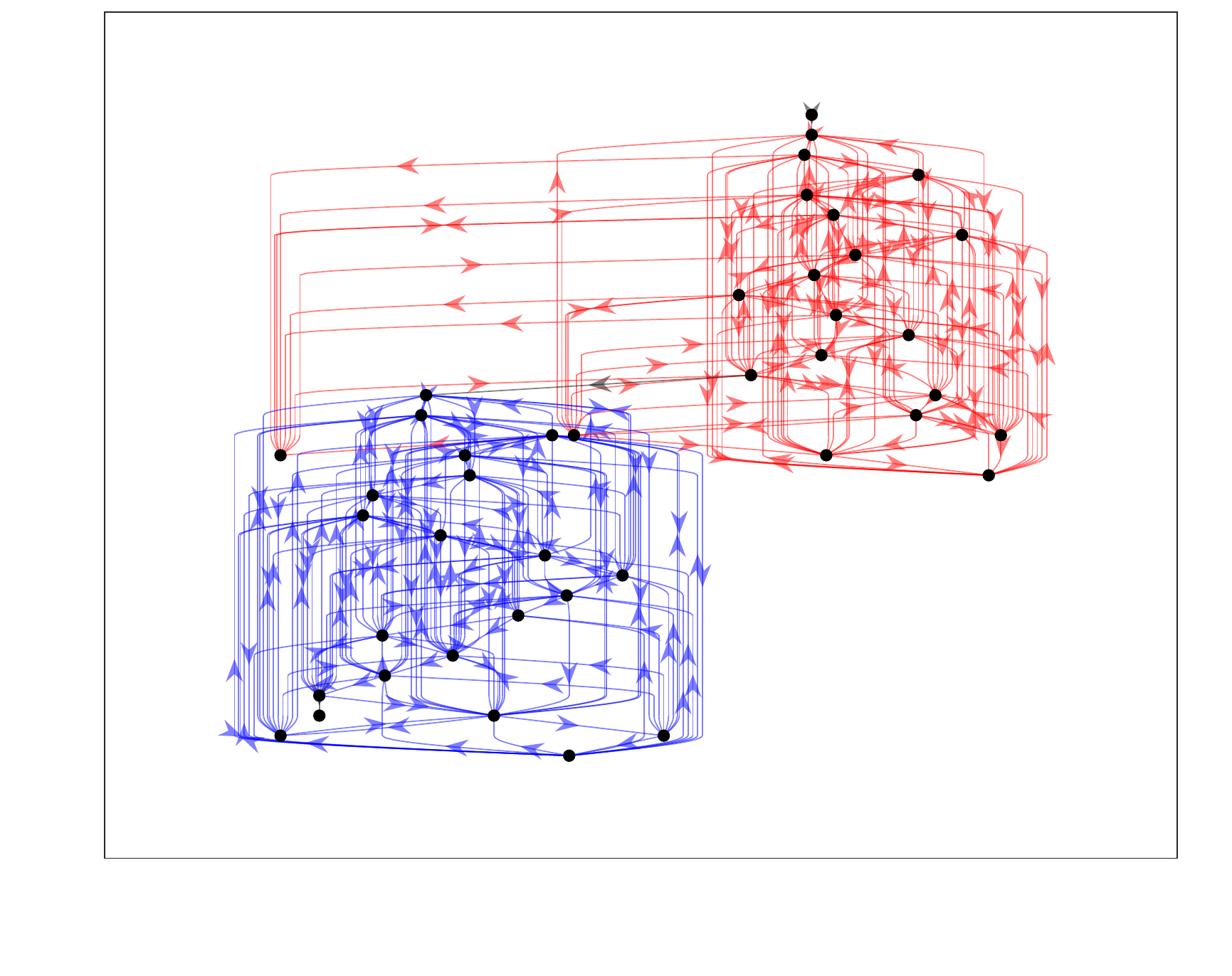}
  \includegraphics[trim = 20mm 65mm 20mm 75mm, clip, width=.49\textwidth,keepaspectratio]{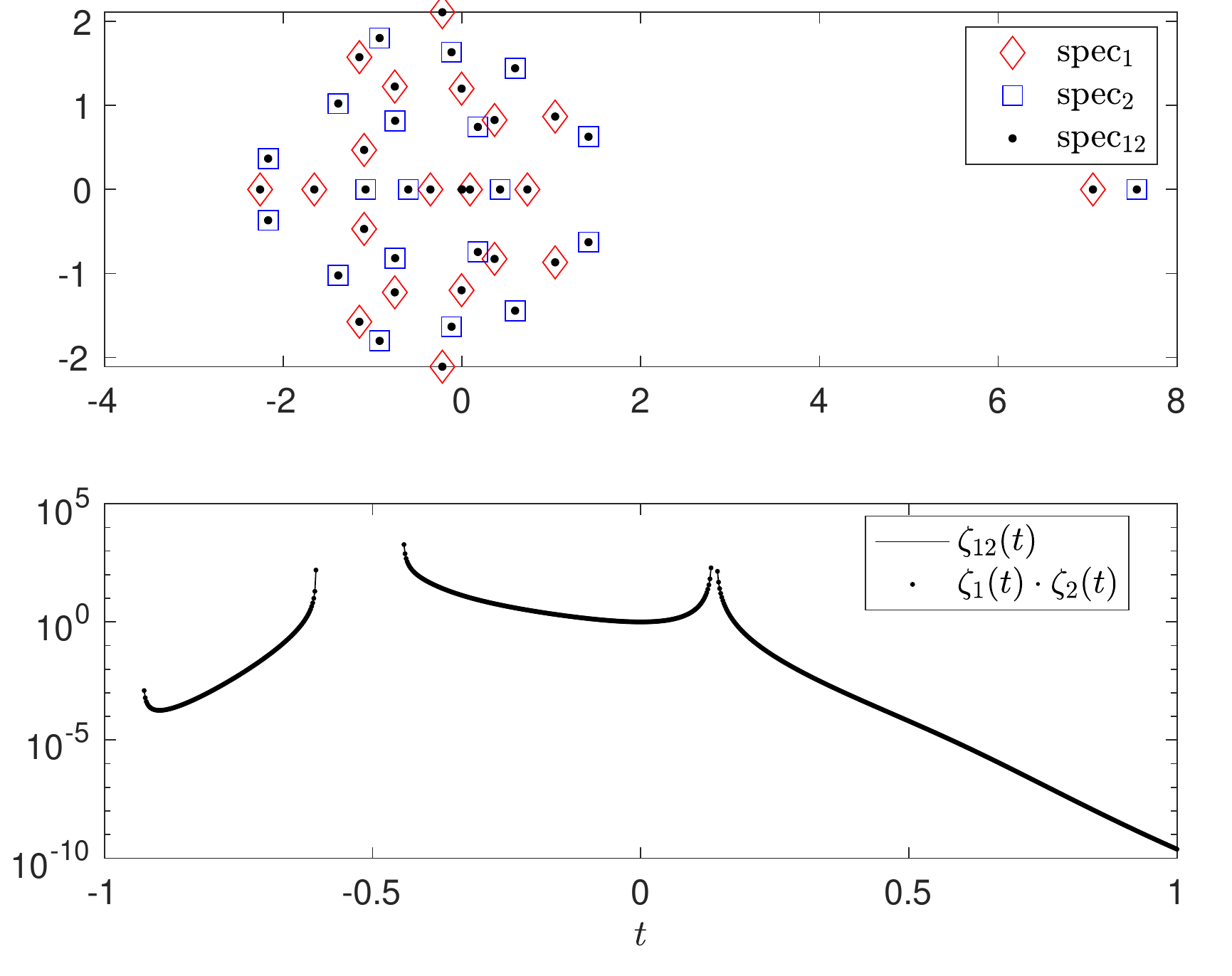} 
  \caption{\label{fig:baz} As in Figure \ref{fig:foo}, but for two flow graphs ${\color{red}D_1}$ and ${\color{blue}D_2}$ on 20 vertices.}
\end{figure}

\end{remark}

Recall that $\max$ furnishes a monoidal structure on the poset $([0,\infty], \ge)$ of extended nonnegative real numbers, and that categories enriched over this are Lawvere ultrametric spaces \cite{nlab:enrichedCategory}. Similarly, $([-\infty,\infty), \le, -\infty, \max)$ is a monoidal poset. This is sufficient data for us to define (following \cite{leinster2017magnitude}) the \emph{magnitude} of ${\bf SubFlow}_D$ over the \emph{max-plus or tropical semiring} \cite{heidergott2014max,nlab:maxPlusAlgebra}. 
\footnote{
It is important to distinguish between the magnitude of ${\bf SubFlow}_D$ as an enriched category and the magnitude of $D$ as a digraph with the usual (asymmetric) notion of distance. Here we are concerned only with the former.
}

Unpacking the details, we have the similarity matrix
\begin{equation}
(Z^\boxtimes_D)_{st} \equiv Z^\boxtimes_D(e_s,e_t) := h(D\langle e_s, e_t \rangle).
\end{equation}
Now if there exist $v$, $w$ satisfying the max-plus matrix (co)weighting equations
\begin{equation}
\max_s [v_s + (Z^\boxtimes_D)_{st} ] = 0 = \max_t [(Z^\boxtimes_D)_{st} + w_t], \nonumber
\end{equation}
then the maxima of $v$ and $w$ coincide and also equal the magnitude of $Z^\boxtimes_D$. Such linear equations can be solved via methods described in \cite{heidergott2014max}, and we simply report the result here: the unique ``principal solutions'' (which may not be \emph{bona fide} solutions in general) are $\hat v_s := -\max_t (Z^\boxtimes_D)_{st}$; $\hat w_t := -\max_s (Z^\boxtimes_D)_{st}$. We therefore obtain the following

\begin{lemma} 
\label{lemma:linear} 
$Z^\boxtimes_D$, and hence ${\bf SubFlow}_D$, has well-defined magnitude $z$ over the max-plus semiring iff 
\begin{equation}
\label{eq:linear}
\max_s [-\max_t (Z^\boxtimes_D)_{st}] = z = \max_t [-\max_s (Z^\boxtimes_D)_{st}]. \qed
\end{equation}
\end{lemma}

It is not obvious when such a $z$ can exist. However, by Lemma 5 of \cite{huntsman2019multiresolution}, any nontrivial $D\langle e_s, e_t \rangle$ must be of the form $\boxtimes_j D \langle e_{j-1},e_j \rangle$ where the $D \langle e_{j-1},e_j \rangle$ are minimal. Appealing to Lemma \ref{lemma:entropy}, we therefore obtain the following

\begin{theorem} 
\label{theorem:flowMagnitude} 
$Z^\boxtimes_D$, and hence ${\bf SubFlow}_D$, has well-defined magnitude over the max-plus semiring. \qed
\end{theorem}

\begin{example}
\label{ex:bundle}
Consider a flow graph of the form $D := \otimes_{k = 1}^K \boxtimes_{j = 1}^{J_k} D \langle e_{(j-1,k)}, e_{(j,k)} \rangle$, where $\otimes$ denotes the parallel tensor/composition on ${\bf Flow}$ described in \cite{huntsman2019multiresolution}. For an example, see Figure \ref{fig:bundle}. For convenience, further assume that the program structure trees of $D \langle e_{(j-1,k)}, e_{(j,k)} \rangle$ are all trivial, i.e., there are no nontrivial sub-flow graphs. Then $(Z^\boxtimes_D)_{(j_0,k),(j_1,k)} = \max_{j_0 < j \le j_1} h (D \langle e_{(j-1,k)}, e_{(j,k)} \rangle)$, $(Z^\boxtimes_D)_{-\infty,\infty} = h(D)$, where $\mp \infty$ indicate the entry and exit edges of $D$, and all other entries of $Z^\boxtimes_D$ are trivial. 

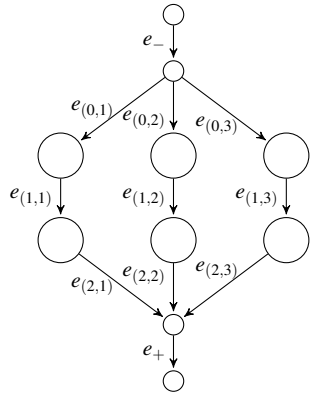
\begin{figure}
	\begin{center}
	\begin{tikzpicture}[scale=0.75,->,>=stealth',shorten >=1pt,every node/.style={transform shape}]
		\node [draw,circle,minimum size=2mm] (vN) at (0,2.5) {};
		\node [draw,circle,minimum size=2mm] (v0) at (0,1.5) {};
		\node [draw,circle,minimum size=8mm] (v11) at (-2,0) {};
		\node [draw,circle,minimum size=8mm] (v21) at (-2,-1.5) {};
		\node [draw,circle,minimum size=8mm] (v12) at (0,0) {};
		\node [draw,circle,minimum size=8mm] (v22) at (0,-1.5) {};
		\node [draw,circle,minimum size=8mm] (v13) at (2,0) {};
		\node [draw,circle,minimum size=8mm] (v23) at (2,-1.5) {};
		\node [draw,circle,minimum size=2mm] (v1) at (0,-3) {};
		\node [draw,circle,minimum size=2mm] (vP) at (0,-4) {};
		\path[->] (vN) edge node [left] {$e_-$} (v0);
		\path[->] (v0) edge node [left] {$e_{(0,1)}$} (v11); 
			\path[->] (v11) edge node [left] {$e_{(1,1)}$} (v21); 
			\path[->] (v21) edge node [left] {$e_{(2,1)}$} (v1);
		\path[->] (v0) edge node [left,pos=0.75] {$e_{(0,2)}$} (v12); 
			\path[->] (v12) edge node [left] {$e_{(1,2)}$} (v22); 
			\path[->] (v22) edge node [left,pos=0.25] {$e_{(2,2)}$} (v1);
		\path[->] (v0) edge node [left,pos=0.75] {$e_{(0,3)}$} (v13); 
			\path[->] (v13) edge node [left] {$e_{(1,3)}$} (v23); 
			\path[->] (v23) edge node [left,pos=0.25] {$e_{(2,3)}$} (v1);
		\path[->] (v1) edge node [left] {$e_+$} (vP);
	\end{tikzpicture}
	\end{center}
	\caption{\label{fig:bundle} Flow graph of the form $D := \otimes_{k = 1}^K \boxtimes_{j = 1}^{J_k} D \langle e_{(j-1,k)}, e_{(j,k)} \rangle$ for $J_k \equiv 2$ and $K = 3$. The large nodes indicate nontrivial interiors of sub-flow graphs.}
\end{figure}

The nontrivial weighting components are therefore 
\begin{equation}
w_{(j,k)} = -\max_{j_0 < j} h (D \langle e_{(j_0,k)}, e_{(j,k)} \rangle) = -\max_{j_0 < j} h (D \langle e_{(j_0,k)}, e_{(j_0+1,k)} \rangle), \nonumber
\end{equation}
while the nontrivial coweighting components are 
\begin{equation}
v_{(j,k)} = -\max_{j_1 > j} h (D \langle e_{(j,k)}, e_{(j_1,k)} \rangle) = -\max_{j_1 > j} h (D \langle e_{(j_1-1,k)}, e_{(j_1,k)} \rangle). \nonumber
\end{equation}
That is, the weighting and coweighting respectively encode the cumulative forward and reverse maxima of the topological entropy along the $K$ ``backbones'' $\boxtimes_{j = 1}^{J_k} D \langle e_{(j-1,k)}, e_{(j,k)} \rangle$ of $D$. In particular, $v_{(j_*-1,k)} = w_{(j_*,k)}$ when $j_* = \arg \max_j h (D \langle e_{(j-1,k)}, e_{(j,k)} \rangle)$.
\end{example}

Finally, it is evident that similar behavior to that detailed in Example \ref{ex:bundle} should occur when each $D \langle e_{(j-1,k)}, e_{(j,k)} \rangle$ is itself of the form $\otimes_m \boxtimes_\ell D' \langle e'_{(\ell-1,m)}, e'_{(\ell,m)} \rangle$, and so on. That is, (co)weightings reliably encode salient features for ``series-parallel'' flow graphs. It seems likely that the same is true for flow graphs that correspond to ``structured'' control flow, which can always be obtained from ``unstructured'' control flow \cite{zhang2004using} in the event that it makes any practical difference.

Operationally, the (co)weighting identifies regions of high topological entropy.
\footnote{
NB. Both $Z^\boxtimes_D$ and its (co)weighting are efficiently computable, as is any necessary preprocessing/restructuring of $D$.
}
This echoes the observations of \cite{bunch2020practical} that (co)weightings pick out salient features of Euclidean point clouds (e.g., ``strata'' of sampled psuedomanifolds). In turn, this suggests a strategy for ``anchoring'' graph matching methods for related flow graphs (e.g., for different versions of the same program or business process). Namely, iteratively coarsen suitably (re)structured flow graphs using the technique of \cite{huntsman2019multiresolution}, attempting to match regions of high topological entropy at each stage of the process.
Recalling Example \ref{ex:bundle}, suppose that $D' := \otimes_{k = 1}^K \boxtimes_{j = 1}^{J_k} D' \langle e'_{(j-1,k)}, e'_{(j,k)} \rangle$ is somehow related to $D$. 
We can hope to leverage the respective (co)weightings for graph matching between $D$ and $D'$.

\section{\label{sec:Neighborhoods}Magnitudes of balls in the universal cover of a digraph}

For a finite strong digraph, a ball around any vertex (defined by, e.g. distance to or from that vertex) eventually saturates. It is helpful to shift perspectives to the \emph{universal cover} \cite{dorfler1980covers} to avoid this saturation while using a notion of the size of these balls to characterize the digraph.
\footnote{
For the conventional notion of a universal cover in topology, see \cite{hatcher2002algebraic,ghrist2014elementary}.
}
This perspective shift is motivated by the context of a (compact connected) Riemannian manifold, for which the \emph{volume entropy} \cite{manning1979topological} is defined via $\lim_{r \uparrow \infty} r^{-1} \log \text{vol} (B_x(r))$, where $B_x(r)$ is the ball of radius $r$ around a point $x$ in the universal cover of the manifold. It turns out that the volume entropy is independent of the point $x$. Also, the volume entropy is bounded above by the topological entropy of the geodesic flow, with equality in the case of nonpositive sectional curvature. Proposition \ref{prop:volumeEntropy} is a very close analogue of this result.
\footnote{
There is a kind of volume entropy for metric graphs \cite{lim2008minimal,lee2019volume} (see also \cite{karam2015growth}), but we are unaware of a digraph analogue.
}

Returning to the context of digraphs, the universal cover of a digraph is a \emph{polytree}, (i.e., an acyclic digraph whose corresponding undirected graph is a tree) that ``locally looks like the digraph everywhere.'' A telling advantage of this construction is that (at the cost of implicitly encoding structure) it renders the calculation of magnitude functions trivial:

\begin{lemma}
\label{lem:polyforestMagnitude}
Let $F$ be a \emph{polyforest}, i.e., an acyclic digraph whose corresponding undirected graph is a forest. Then the magnitude function of $F$ (i.e., the magnitude of $\exp[-td]$ where $d$ is the usual Lawvere metric on $F$) is $|V(F)| - |E(F)| e^{-t}$.
\footnote{Note that if $F$ is a polytree, then $|V(F)| = |E(F)|+1$.} 
\end{lemma}

\begin{proof}[Proof (sketch).] The proof can be adapted almost wholesale from an analogous result for undirected trees (or for that matter, forests) in \S 4 of \cite{leinster2019magnitude}: apart from checking and slightly adjusting definitions, the key observation is that the magnitude function of a digraph with a single (directed) edge is $2-e^{-t}$ (by comparison, the magnitude function of a graph with a single edge is $2(1+e^{-t})^{-1}$).
\end{proof}

The universal cover $U_D := (V_U,E_U)$ of a weak digraph $D = (V,E)$ is a polytree defined as follows \cite{dorfler1980covers}: pick $v_0 \in V$ and set
\begin{equation}
V_U := \{(v_0,v_1,\dots,v_L) : (v_{j-1},v_j) \in E; v_{j-1} \ne v_j\} \cup \{(v_L,v_{L-1},\dots,v_0) : (v_j,v_{j-1}) \in E; v_j \ne v_{j-1}\} \nonumber
\end{equation}
where $v_j \in V$ and $e_j \in E$ identically; and set 
\begin{align}
E_U := \ & \{((v_0,v_1,\dots,v_{L-1}),(v_0,v_1,\dots,v_L)) : (v_{L-1},v_L) \in E \} \nonumber \\ 
& \cup \{((v_0,v_1,\dots,v_L),(v_0,v_1,\dots,v_{L-1})) : (v_L,v_{L-1}) \in E \}. \nonumber
\end{align}

\begin{example}
\label{ex:cover}
Consider the digraph $D$ in the left panel of Figure \ref{fig:cover}. Its universal cover has local structure shown in the right panel of Figure \ref{fig:cover}, and the covering map is depicted in Figure \ref{fig:coverProjection} (which also shows a larger local region of the universal cover).

\begin{figure}[h]
  \centering
  \includegraphics[trim = 92mm 120mm 88mm 120mm, clip, width=.15\textwidth,keepaspectratio]{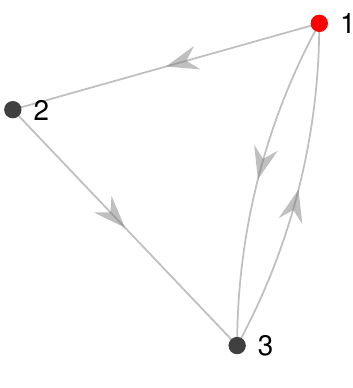} \quad \quad
  \includegraphics[trim = 63mm 120mm 46mm 120mm, clip, width=.5\textwidth,keepaspectratio]{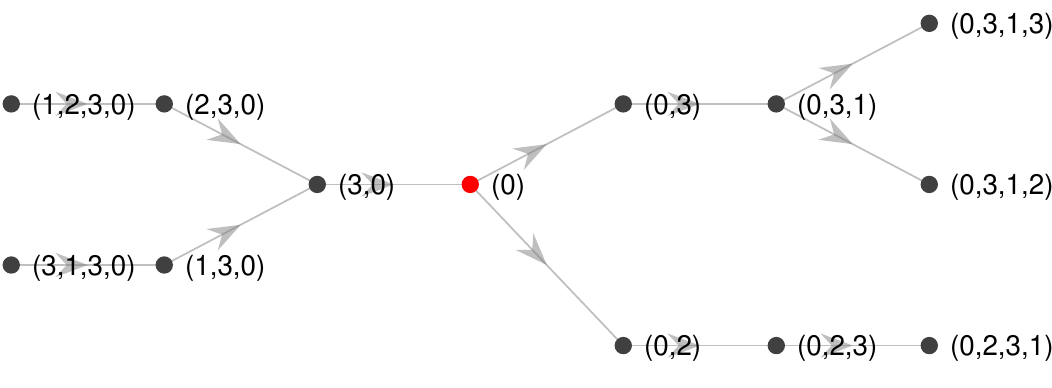}
  \caption{\label{fig:cover} (L) A strong loopless digraph $D$ with {\color{red}basepoint $v_0 = 1$ highlighted in red}. (R) The portion of $U_D$ with vertices at distance $\le 3$ to or from {\color{red}$v_0$}. Vertices of $U_D$ are labeled by the corresponding sequence of $D$-vertices, with $0$ explicitly indicating the basepoint. The ball $B_0(3)$ is formed by taking the arborescence of depth $3$ rooted at $0$, i.e., the right-hand part.}
\end{figure}

\begin{figure}[h]
  \centering
  \includegraphics[trim = 88mm 105mm 65mm 105mm, clip, width=.29\textwidth,keepaspectratio]{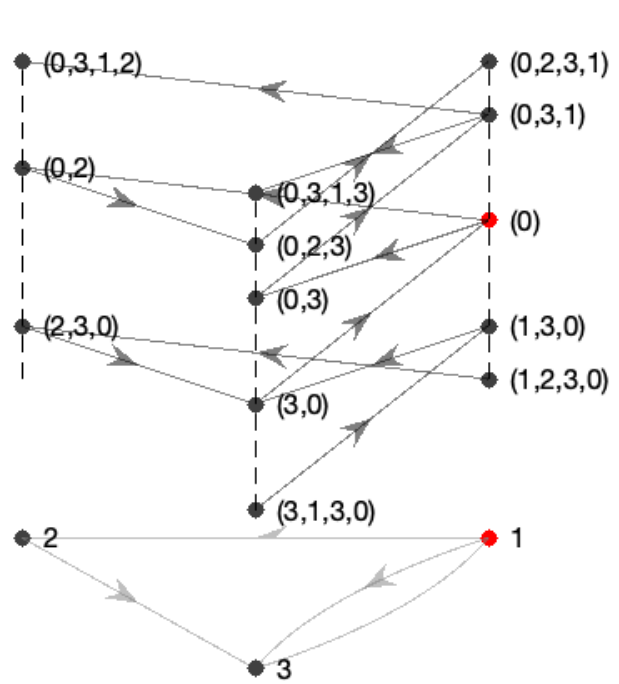}
  \includegraphics[trim = 60mm 120mm 55mm 120mm, clip, width=.69\textwidth,keepaspectratio]{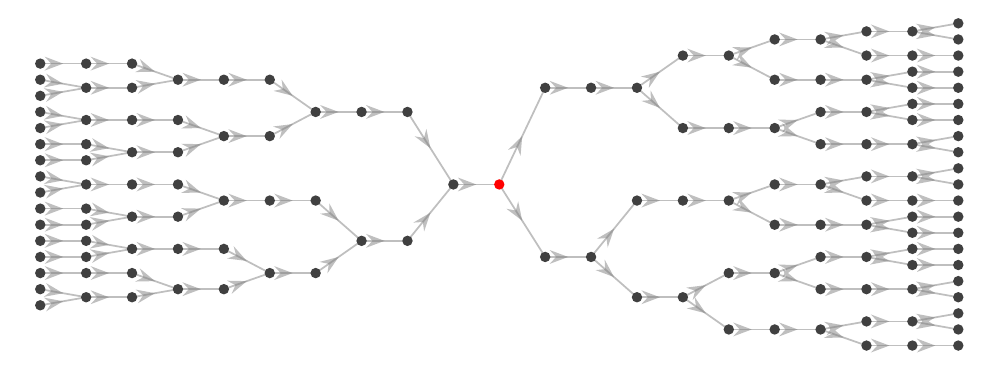}
  \caption{\label{fig:coverProjection} (Cf. Figure \ref{fig:cover}.) (L) The portion of $U_D$ with vertices at distance $\le 3$ to or from {\color{red}$v_0$} with covering of $D$ (at bottom) indicated. (R) The portion of $U_D$ with vertices at distance $\le 10$ to or from {\color{red}$v_0$}.}
\end{figure}

\end{example}

\begin{proposition}
Let $\gamma \in V_U$. Then there is either a unique path in $U_D$ from $v_0$ to $\gamma$ or \emph{vice versa}. \qed
\end{proposition}

The number of paths from $v_0$ of length $L$ in $U_D$ equals the number of loopless paths from $v_0$ of length $L$ in $D$. Define $B_{v_0}(L)$ to be the sub-polytree of $U_D$ (defined with basepoint $v_0$) induced by its vertices at (the usual notion of digraph) distance $\le L$ from (versus to) $v_0$. We can compute the magnitude function of $B_{v_0}(L)$ very easily using the following proposition.

\begin{proposition}
If $D$ is loopless, then $B_{v_0}(L)$ is an arboresence with $|V(B_{v_0}(L))| = \sum_{\ell = 0}^L \sum_k (A^\ell)_{jk}$, where $A$ is the adjacency matrix of $D$ and $j$ is the matrix index corresponding to $v_0$. \qed
\end{proposition}

\begin{remark}
By comparison, the \emph{Katz centrality} is $\sum_{\ell = 1}^\infty \alpha^\ell \sum_i (A^\ell)_{ij}$, where $\alpha$ is restricted to ensure convergence \cite{grindrod2013matrix}. The Katz centrality of the graph with all edges reversed is therefore $\sum_{\ell = 1}^\infty \alpha^\ell \sum_k (A^\ell)_{jk}$.
\end{remark}

Since an arborescence (or more generally a polytree) has one more vertex than it has edges, Lemma \ref{lem:polyforestMagnitude} yields that for $D$ loopless, the magnitude function of $B_{v_0}(L)$ is 
\begin{equation}
\operatorname{Mag} (B_{v_0}(L),t) = |V(B_{v_0}(L))| - (|V(B_{v_0}(L))|-1)e^{-t},
\end{equation}
and the most recent proposition gives an elementary algorithm for computing $|V(B_{v_0}(L))|$. If $D$ is loopless and strong, we have $h(D) = h(U_D) =: \lim_{L \uparrow \infty} L^{-1} \log |V(B_{v_0}(L))|$ independent of the basepoint $v_0$. 

\begin{proposition}
\label{prop:volumeEntropy}
Let $D$ be a strong loopless digraph and $v_0 \in V(D)$. Then 
\begin{equation}
\lim_{L \uparrow \infty} L^{-1} \log \operatorname{Mag} (B_{v_0}(L),t) \le h(D)
\end{equation}
with equality at $t = \infty$, and the left hand side is independent of $v_0$ for any $t$. Here $\operatorname{Mag} (\cdot,t)$ denotes the magnitude function of the first argument. \qed
\end{proposition}

\begin{example}
\label{ex:cover2}
Continuing Example \ref{ex:cover}, $h(D) \approx 0.2812$ is the logarithm of the so-called \emph{plastic number}, i.e., the unique real solution of $x^3-x-1 = 0$. Numerics suggest that $|V(B_1(L))|$ is given by \cite{oeisA167385}. Assuming this to obtain values for large $L$, we show the convergence of $L^{-1} \log \operatorname{Mag} (B_{v_0}(L),t)$ in Figure \ref{fig:coverConvergence}.

\begin{figure}[h]
  \centering
  \includegraphics[trim = 45mm 109mm 40mm 106mm, clip, width=.55\textwidth,keepaspectratio]{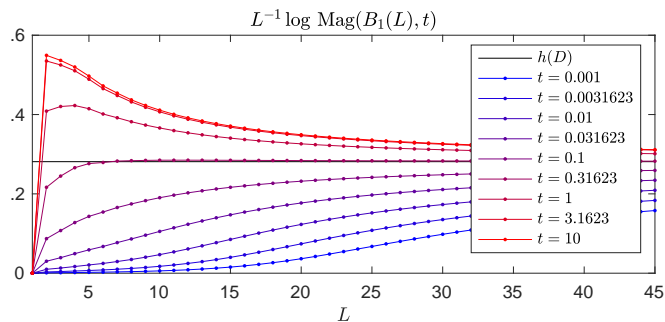}
  \caption{\label{fig:coverConvergence} $L^{-1} \log \operatorname{Mag} (B_{v_0}(L),t) \rightarrow h(D)$ for $t > 0$, but depends strongly on $t$ even for fairly large $L$.}
\end{figure}

\end{example}

\section{\label{sec:Example}Example: correlated features for digraph matching}

In this section we detail how log-magnitudes of small balls associated to the Lawvere metric structure on a digraph are both interesting and useful from the perspective of feature engineering; for completeness and comparison, we start by considering the ambient (co)weighting. In keeping with the general theme of providing tools for graph matching, we focus on the import graph of the \texttt{Flare} software hierarchy, accessed from \url{https://observablehq.com/@d3/hierarchical-edge-bundling/2} in November 2020 and depicted in Figure \ref{fig:flare}. 

\begin{figure}[h]
  \centering
  \includegraphics[trim = 0mm 50mm 0mm 40mm, clip, width=\textwidth,keepaspectratio]{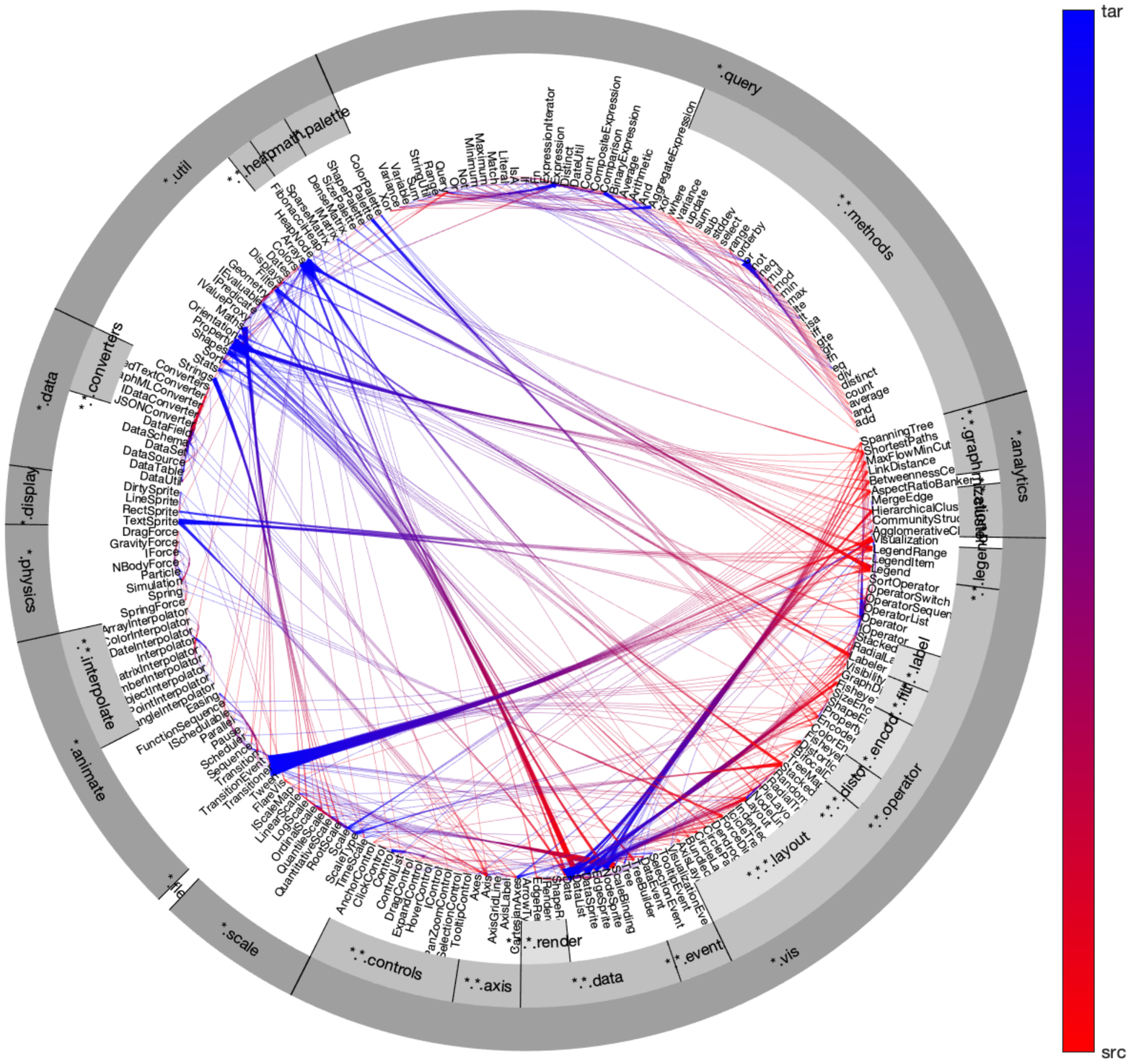}
  \caption{Import graph of the \texttt{Flare} software hierarchy, displayed using the divided edge bundling approach of \cite{selassie2011divided}. Edge {\color{red}sources} and {\color{blue}targets} are respectively colored {\color{red}red} and {\color{blue}blue}; the hierarchy is depicted along the figure periphery (and without any material loss of information from occlusions).}
  \label{fig:flare}
\end{figure}

As an experiment, we considered $N = 100$ realizations of a pair of random subgraphs of the ``ambient'' digraph of Figure \ref{fig:flare} obtained by removing edges with probability $3/4$ and then retaining the largest weak component. 
We then computed the (co)weightings at scale 0, the log-magnitudes of balls of radius $\le 3$ at scale $t = 100$ (which is virtually equivalent to $t = \infty$), and various common vertex centrality measures. For each of these quantities and $N$ realizations, we then computed the correlation coefficients on vertices shared by the pair of subgraphs. The results are shown in Figure \ref{fig:centralitiesFlare}, which shows that the coweighting and log-magnitudes of balls in the universal cover of the digraph with edges reversed are very strongly correlated. This suggests the utility of such features for graph comparison \cite{wills2020metrics} and matching \cite{emmert2016fifty}.
\footnote{
For a naive approach of matching nodes based on rankings derived from centralities, see Figure 7 of \cite{sato2020fast}.
}

\begin{figure}[h]
  \centering
  \includegraphics[trim = 20mm 85mm 20mm 85mm, clip, width=.8\textwidth,keepaspectratio]{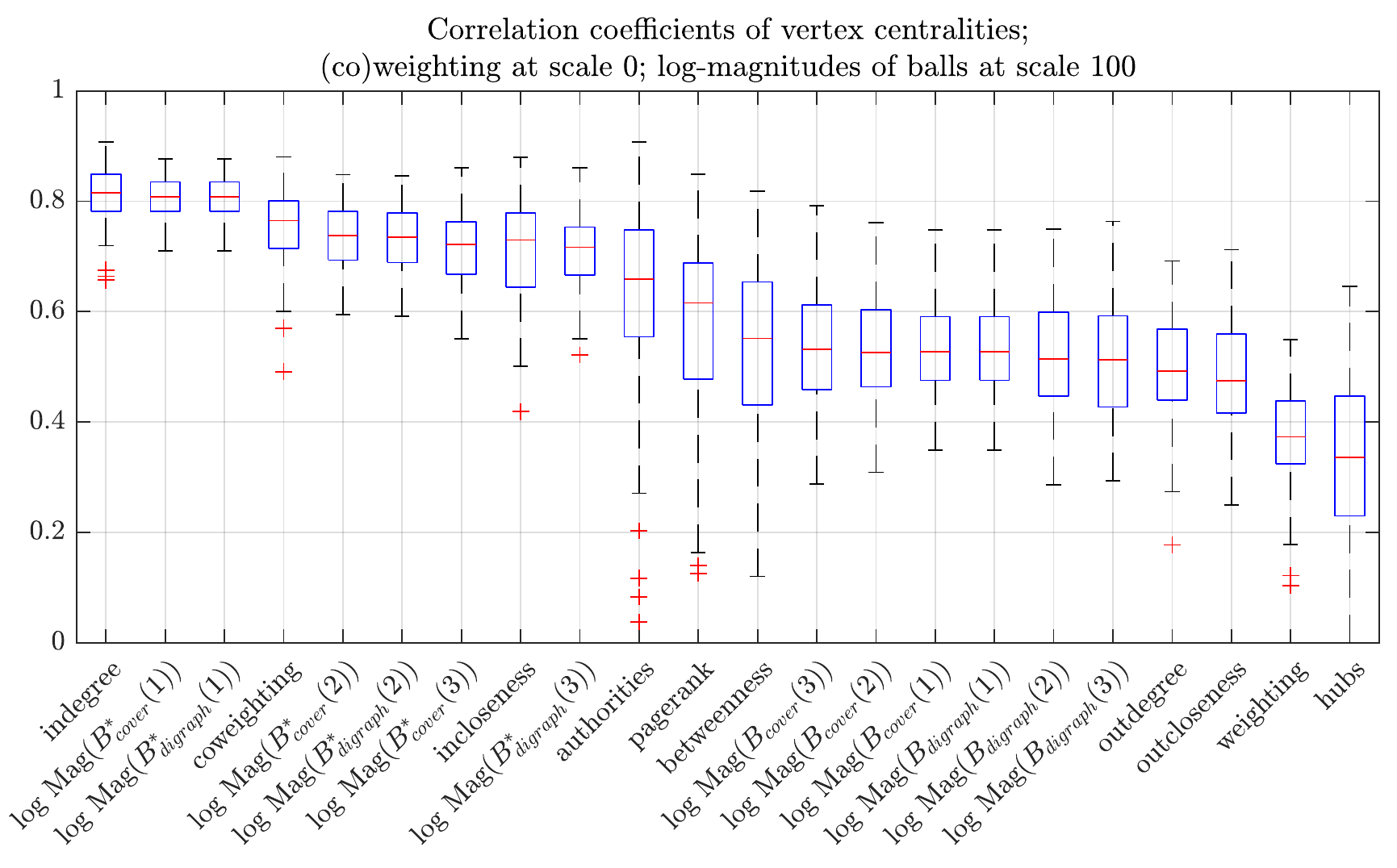}
  \caption{Distribution of correlations for various centralities between two random subgraphs of the digraph in Figure \ref{fig:flare}. * indicates a ball in the digraph with all edges reversed. As $L$ increases, boundary effects cause the log-magnitudes of balls in the universal cover to become (slightly) more correlated to each other than the log-magnitudes of balls in the digraph itself. Note that the three best-performing centralities are computing almost exactly the same thing.}
  \label{fig:centralitiesFlare}
\end{figure}

The strong correlations of log-magnitudes of balls are more robust than those of (co)weightings, as an experiment along the same lines as above but using different realizations of an Erd\H{o}s-Reny\'i digraph ($n = 100$ vertices; edge probability $q = 4/n$) as the ambient digraph for each of $N = 100$ trials shows. We formed two subgraphs by removing edges with probability $1/2$, then retaining the largest weak component. Figure \ref{fig:centralitiesER} shows the results, which are qualitatively echoed for different parameters.

\begin{figure}[h]
  \centering
  \includegraphics[trim = 20mm 85mm 20mm 85mm, clip, width=.8\textwidth,keepaspectratio]{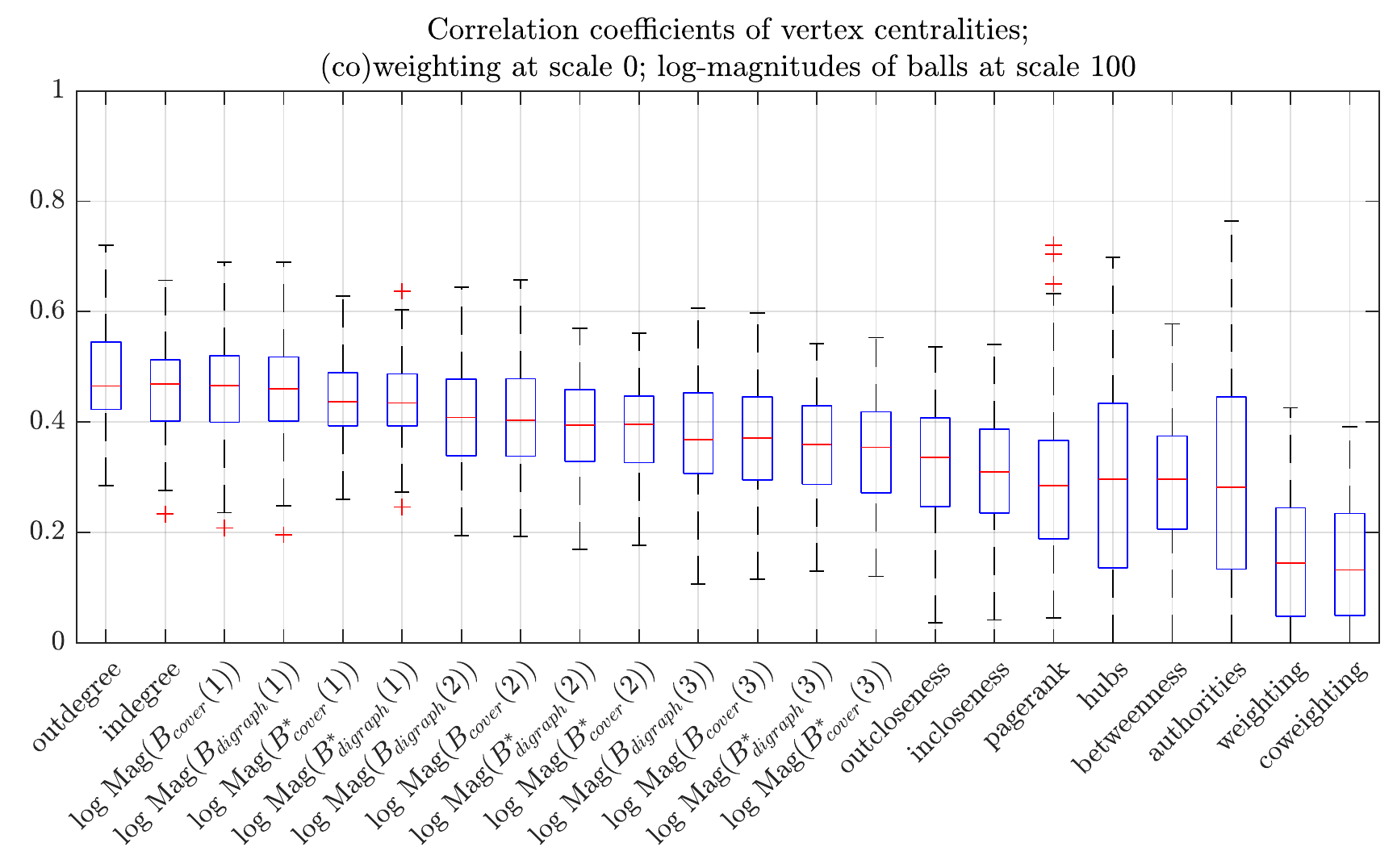}
  \caption{Distribution of correlations for various centralities between two random subgraphs of distinct Erd\H{o}s-Reny\'i digraphs with $n = 100$ vertices and edge probability $q = 4/n$; cf. Figure \ref{fig:centralitiesFlare}.}
  \label{fig:centralitiesER}
\end{figure}

One theoretical advantage of using log-magnitudes of balls is that unlike (co)weightings, these are nonnegative by construction.
\footnote{
NB. It is possible using elementary bounds to efficiently determine the minimal $t$ such that a [co]weighting of $Z = \exp[-td]$ is nonnegative \cite{huntsman2022diversity}. 
} 
\footnote{
Note that from the point of view of correlation analyses, log-magnitudes are more interesting than the magnitudes themselves: because the correlation coefficient is invariant under affine transformations of either argument, we have that $\rho(\operatorname{Mag}(B_{v_0}(L),t), \operatorname{Mag}(B'_{v_0}(L),t')) = \rho(|V(B_{v_0}(L))|,|V(B'_{v_0}(L))|)$. In any event, $\lim_{t \uparrow \infty} \operatorname{Mag}(B_{v_0}(L),t) = |V(B_{v_0}(L))|$.
}
This may be advantageous in the context of graph matching via optimal transport techniques that require a sensible distribution on vertices. In particular, the recently developed \emph{Gromov-Wasserstein distance} \cite{memoli2011gromov,memoli2014gromov} is useful for analyzing weighted digraphs endowed with measures \cite{chowdhury2019gromov} and has been applied to (mostly but not exclusively undirected) graph matching with state of the art performance \cite{vayer2019optimal,xu2019gromov,xu2019scalable,chowdhury2020generalized,chowdhury2020gromov,vayer2020fused}.
For instance, although \cite{xu2019scalable} did not consider digraphs, it used a distribution proportional to $(\deg + a)^b$, where $a$ and $b$ are hyperparameters, and remarked that ``the node distributions have a huge influence on the stability and the performance of our learning algorithms.'' Meanwhile, this particular sort of distribution is rather similar to the log-magnitude of a unit ball for $a = 1$ and $b = 0$. In short, we can plausibly expect to improve upon the approach of \cite{xu2019scalable} in the context of digraphs by using weightings rather than a more \emph{ad hoc} distribution.

\section*{Acknowledgement}

This research was developed with funding from the Defense Advanced Research Projects Agency (DARPA). The views, opinions and/or findings expressed are those of the author and should not be interpreted as representing the official views or policies of the Department of Defense or the U.S. Government. Distribution Statement ``A'' (Approved for Public Release, Distribution Unlimited)

\bibliographystyle{eptcs}

\end{document}